\documentclass{amsart}
\usepackage{hyperref}
\usepackage{amsmath,diagbox,amssymb,tikz}
\usetikzlibrary{arrows,shapes,positioning,decorations.markings}
\newtheorem{theorem}{Theorem}[section]
\newtheorem{lemma}[theorem]{Lemma}
\newtheorem{corollary}[theorem]{Corollary}

\newtheorem{remark}[theorem]{Remark}

\numberwithin{equation}{section}

\newcommand{\stirlingii}{\genfrac{[}{]}{0pt}{}}

\makeatletter
\newcommand{\myitem}[1]{%
\item[#1]\protected@edef\@currentlabel{#1}%
}
\makeatother
\begin{document}
\title[Bounding the number of subgroups]{An explicit upper bound on the number of subgroups of a finite group} 

\author[P. Spiga]{Pablo Spiga}
\address{Pablo Spiga,
Dipartimento di Matematica e Applicazioni, University of Milano-Bicocca,\newline
Via Cozzi 55, 20125 Milano, Italy}\email{pablo.spiga@unimib.it}


\begin{abstract}
In this paper we prove that a finite group of order $r$ has at most
$$ 7.3722\cdot r^{\frac{\log_2r}{4}+1.5315}$$
subgroups. 
\end{abstract}

\keywords{subgroups, bound}
\subjclass[2010]{20D99}
\maketitle

\section{Introduction}
Let $R$ be a finite group of cardinality $r$. Since a chain of subgroups of $R$ has length at most $\log_2r$, we deduce that every subgroup $H$ of $R$ has a generating set of cardinality at most $\lfloor\log_2r\rfloor$. In particular, since the generators of $H$ are elements of $R$, we have at most $$r^{\lfloor\log_2r\rfloor}\le r^{\log_2r}$$
choices for $H$. In other words, $R$ has at most $r^{\log_2r}$ subgroups. This is the typical argument for bounding the number of subgroups of a finite group $R$.

This bound is not too far off from the best possible. In fact, an elementary abelian $2$-group of order $r:=2^a$ has
\[
\stirlingii{a}{\lfloor \frac{a}{2}\rfloor}_2:=
\frac{2^a-1}
{2^{\lfloor\frac{a}{2}\rfloor}-1}\frac{2^{a-1}-1}{2^{\lfloor\frac{a}{2}\rfloor-1}-1}\cdots\frac{2^{a-\lfloor\frac{a}{2}\rfloor+1}-1}{2-1}
\]
subgroups of order $2^{\lfloor a/2\rfloor}$. Since $\stirlingii{a}{\lfloor \frac{a}{2}\rfloor}_2\ge 2^{a^2/4}$, $R$ has at least $2^{a^2/4}=r^{\log_2r/4}$ subgroups. 

Borovik, Pyber and Shalev~\cite[Corollary~$1.6$]{BPS} have shown that a finite group of order $r$ has at most $$r^{\log_2(r)\cdot \left(\frac{1}{4}+o(1)\right)}$$ subgroups. Therefore, in the light of the previous paragraph, this bound is somehow best possible.

In applications, however, it is sometimes useful to have an explicit upper bound rather than an asymptotic result. For instance, in proving the Babai-Godsil conjecture~\cite{MorrisSpiga} on the asymptotic enumeration of Cayley digraphs, the authors have used numerous times the trivial bound $r^{\log_2r}$ in their argument. However, there  are applications where the naive bound $r^{\log_2 r}$ does not suffice and the $1/4$ improvement in the exponent might give considerable help. For instance, this turns out useful in investigating the asymptotic enumeration of Haar graphs~\cite{Spiga}.

The fundamental part of the argument in~\cite{BPS} comes from an estimate on the number of maximal subgroups of a finite group. The papers~\cite{B,LPS} do obtain very useful information on the number of maximal subgroups in a finite group; however, the implicit constants make it difficult to extract  explicit bounds. Analogously, the bounds in~\cite{BEJ} are very explicit, but they do depend on the number of generators of $R$. Hence, it seems cumbersome to use~\cite{BEJ} for obtaining an explicit upper bound on the number of subgroups of a finite group.

In this paper we prove the following result.

\begin{theorem}\label{thrm:main}
A finite non-identity group  $R$ has at most $7.3722\cdot|R|^{\frac{\log_2|R|}{4}+1.5315}$ subgroups.
\end{theorem}
Our group theoretic argument is entirely elementary and we give a sketch of the proof in Section~\ref{group}. This part is essentially as in~\cite{BPS}. However, rather than reducing to certain maximal subgroups, we make some detailed arithmetic considerations that allow us to prove Theorem~\ref{thrm:main}. We do believe that similar considerations can prove the upper bound $7.3722\cdot|R|^{\log_2|R|}$; however, this seems to require (at least with our method) some very long arithmetic arguments.

\section{A group theoretic argument}\label{group}
In this section we give a sketch of our  proof of Theorem~\ref{thrm:main}. 

\begin{proof}[Sketch of the proof of Theorem~$\ref{thrm:main}$: part I]
Let $R$ be a finite group and let $r$ be the order of $R$. Without loss of generality, we may suppose that $r\ge 2$. Now, we factorize $$r=\prod_{i=1}^\ell p_i^{a_i}$$ into prime factors; in particular, $p_1,\ldots,p_\ell$ are distinct primes and $a_i\ge 1$ for each $i\in\{1,\ldots,\ell\}$. Relabeling the indexed set if necessary, we may suppose that $p_1<\cdots<p_\ell$. 

Let $H$ be a subgroup of $R$. Then $H$ is uniquely determined by a family $(Q_i)_{i}$ of Sylow $p_i$-subgroups of $H$, for each $i\in \{1,\ldots,\ell\}$. Each of these  subgroups $Q_i$ is contained in a Sylow $p_i$-subgroup $P_i$ of $R$. Now, from Sylow's theorem, all Sylow $p_i$-subgroups of $R$ are conjugate and hence $R$ has at most $r/p_i^{a_i}$ Sylow $p_i$-subgroups. Therefore, we have at most $$\prod_{i=1}^\ell\frac{r}{p_i^{a_i}}=r^{\ell-1}$$ choices for the $\ell$-tuple $(P_i)_i$.

In~\eqref{variousS} we define a function $S(p,a)$ such that every $p$-group of order $p^a$ has at most $S(p,a)$ subgroups, see Remark~\ref{remark1}. At this point, let us ignore what this function is and let us see what we may deduce.
 When $(P_i)_i$ is given, since $Q_i$ is a subgroup of $P_i$ and since $P_i$ is a $p_i$-group, from the previous paragraph, we see that we have at most 
$$\prod_{i=1}^\ell S(p_i,a_i)$$
choices for the $\ell$-tuple $(Q_i)_i$.

From above, $R$ has at most
\begin{align}\label{eq:AA}
r^{\ell-1}\cdot\prod_{i=1}^\ell S(p_i,a_i)=r^{-1}\prod_{i=1}^\ell r\cdot S(p_i,a_i)
\end{align} 
subgroups.

The bulk of the argument in Section~\ref{arithmetic} is proving that
\begin{equation}\label{ex}rS(p_i,a_i)\le p_i^{a_i\frac{\log_2r}{4}},
\end{equation}
unless $p_i\le 23$. Actually, Section~\ref{arithmetic} proves much more than that and hence there is room for hoping for an improvement on $7.3722\cdot r^{\log_2 r/4+1.5315}$. For instance, in the particular case that $p_1>23$, from~\eqref{eq:AA} and~\eqref{ex} we deduce the stronger upper bound 
$$r^{-1}\prod_{i=1}^\ell rS(p_i,a)\le r^{-1}\prod_{i=1}^\ell p_i^{a_i\frac{\log_2 r}{4}}=r^{\frac{\log_2 r}{4}-1}.$$ Our weaker bound $7.3722\cdot r^{\frac{\log_2r}{4}+1.5315}$ arises from dealing with small primes in the factorization of $r$. 

We postpone the rest of the proof after Section~\ref{arithmetic}.
\end{proof}

\section{Arithmetical observations}\label{arithmetic}

For each prime number $p$, we let
\begin{align}\label{eq:cp}
C(p)&:=\prod_{i\ge 1}\frac{1}{1-\frac{1}{p^{i}}},\\\nonumber
c(p)&:=2.129\cdot C(p).
\end{align}
Now, let $p$ be a prime number and let $a$ be a positive integer. We define
\begin{equation}\label{variousS}
S(p,a):=
\begin{cases}
2&\textrm{when }a:=1,\\
p+3&\textrm{when }a:=2,\\
2p^2+2p+4&\textrm{when }a:=3,\\
p^4+3p^3+4p^2+3p+5&\textrm{when }a:=4,\\
2p^6+2p^5+6p^4+6p^3+6p^2+4p+6&\textrm{when }a:=5,\\
c(p)p^{\frac{a^2}{2}}&\textrm{when }a\ge 6.
\end{cases}
\end{equation}
Strictly speaking we give more details than it is barely necessary for the proof of Theorem~\ref{thrm:main}. We hope that this information can be used in the future for improving the bound in Theorem~\ref{thrm:main}.
\begin{remark}\label{remark1}{\rm Let $P$ be a $p$-group of order $p^a$. We observe here that $P$ has at most $S(p,a)$ subgroups, where $S(p,a)$ is defined in~\eqref{variousS}.

Let $k\in \{0,\ldots,a\}$. Corollary~$4.2$ in~\cite{Shalev} shows that the number of subgroups of $P$ having index $p^k$ is at most
$$\stirlingii{a}{k}_p\le C(p)p^{a(n-k)},$$
where $C(p)$ is defined in~\eqref{eq:cp}. In particular, the number of subgroups of $P$ is at most
\[
\sum_{k=0}^a
\stirlingii{a}{k}_p.
\]
When $a\le 5$, we see with a computation that this summation is exactly $S(p,a)$. For instance, when $a=4$, we have
\begin{align*}
\stirlingii{a}{0}_p+\stirlingii{a}{1}_p+\stirlingii{a}{2}_p+\stirlingii{a}{3}_p+\stirlingii{a}{4}_p&=1+\frac{p^4-1}{p-1}+\frac{(p^4-1)(p^3-1)}{(p^2-1)(p-1)}+\frac{p^4-1}{p-1}+1\\
&=2(p^3+p^2+p+1+2)+p^4+p^3+2p^2+p+1\\
&=S(p,a).
\end{align*}
Assume now that $a\ge 6$. 

Suppose that $a$ is even. Then the number of subgroups of $R$ is at most
\begin{align*}
C(p)\sum_{k=0}^ap^{k(a-k)}&=C(p)\cdot\left(p^{\frac{a^2}{4}}+2\sum_{k=0}^{\frac{a}{2}-1}p^{k(a-k)}\right)=C(p)p^{\frac{a^2}{4}}\left(1+2\sum_{k=1}^{\frac{a}{2}}\frac{1}{p^{k^2}}\right)\\
&\le C(p)p^{\frac{a^2}{4}}\left(-1+2\sum_{k=0}^{\infty}\frac{1}{p^{k^2}}\right)\le 2.129C(p)p^{\frac{a^2}{4}}=S(p,a),
\end{align*}
where the value $2.129$ is obtained by taking $p:=2$ in the infinite sum. Suppose that $a$ is odd. Then the number of subgroups of $R$ is at most
\begin{align*}
C(p)\sum_{k=0}^ap^{k(a-k)}&
=C(p)\cdot 2\sum_{k=0}^{\frac{a-1}{2}}p^{k(a-k)}=C(p)\cdot 2 p^{\frac{a^2-1}{4}}\sum_{k=0}^{\frac{a-1}{2}}\frac{1}{p^{k(k+1)}}\\
&\le C(p)\cdot 2 p^{\frac{a^2-1}{4}}\sum_{k=0}^{\infty}\frac{1}{p^{k(k+1)}}\le 2.53175 C(p)p^{\frac{a^2-1}{4}}\\
&\le 2.129 C(p)p^{\frac{a^2}{4}}=S(p,a),
\end{align*}
where the value $2.53175$ is obtained by taking $p:=2$ in the infinite sum and where $2.129$ is obtained by multiplying $2.53175$ with $2^{-1/4}$.  Summing up, regardless of whether $a$ is odd or even,  $R$ has at most  $S(p,a)$ subgroups.}
\end{remark}

In this section we prove various upper bounds on $S(p,a)$. We give here the first lemma that in part explains the role of $7.3722\cdot$ in the upper bound in Theorem~\ref{thrm:main}.
\begin{lemma}\label{lemma0}Let $p$ be a prime number, let $a$ be a positive integer and let $r$ be a multiple of $p^a$. Then $S(p,a)\le 7.3722\cdot p^{a\frac{\log_2 r}{4}}$.
\end{lemma}
\begin{proof}
Since $r\to\log_2 r$ is monotone increasing, we may suppose that $r=p^a$. In particular, $7.3722\cdot p^{a\frac{\log_2r}{4}}=7.3722\cdot p^{\frac{a^2}{4}\log_2p}$. Now the proof follows by distinguishing various possibilities for $a$. When $a=1$, we have $S(p,1)=2$ and $$7.3722\cdot p^{\frac{a^2}{4}\log_2p}=7.3722\cdot p^{\frac{\log_2p}{4}}\ge 7.3722\cdot 2^{1/4}=9.5136.$$ When $a=2$, $S(p,2)=p+3$ and $$7.3722\cdot p^{\frac{a^2}{4}\log_2p}=7.3722\cdot p^{\log_2p}\ge 7.3722\cdot p;$$ clearly, $p+3\ge 7.3722\cdot p$. The cases $a\in \{3,4,5\}$ are entirely similar.

Suppose $a\ge 6$. Now, $c(p)p^{\frac{a^2}{4}}=S(p,a)\le 7.3722\cdot p^{\frac{a^2}{4}}$ if and only if $c(p)\le 7.3722$. From~\eqref{eq:cp}, $p\mapsto c(p)$ is a monotone decreasing function and hence $c(p)\le c(2)=7.372187<7.3722$.
\end{proof}

\subsection{Dealing with one prime}\label{sub1}

\begin{lemma}\label{lemma1}
Let $p$ be a prime number and let $r$ be a positive multiple of $p$ with $\gcd(p,r/p) =1$. Then $r\cdot S(p,1)\le p^{\frac{\log_2r}{4}}$ unless one of the following holds
\begin{enumerate}
\item\label{lemma1:eq1}$p=23$ and $1\le r/p\le 8$,
\item\label{lemma1:eq2}$p=19$ and $1\le r/p\le 3\,784$,
\item\label{lemma1:eq3}$p\le 17$.
\end{enumerate}
\end{lemma}
\begin{proof}
Here $S(p,1)=2$. 
The proof follows from easy computations, assisted with the computer. When $p\ge 29$, it can be verified that $p^{\log_2(p)/4}>2p$ and $p^{1/4}>2$. Therefore
\begin{align*}
p^{\frac{\log_2 r}{4}}&=p^{\frac{\log_{2}p}{4}}\cdot p^{\frac{\log_2(r/p)}{4}}>2p\cdot p^{\frac{\log_2(r/p)}{4}}\\
&>2p\cdot 2^{\log_2(r/p)}=2p\cdot (r/p)=2r=rS(p,1).
\end{align*}

Suppose now $p<29$. If $p\le 17$, then we obtain part~\eqref{lemma1:eq3}. If $p>17$, then $p\in \{19,23\} $ and parts~\eqref{lemma1:eq1} and~\eqref{lemma1:eq2} follow with a computer assisted computation.
\end{proof}

\begin{lemma}\label{lemma2}
Let $p$ be a prime number and let $r$ be a positive multiple of $p^2$ with $\gcd(p,r/p^2)=1$. Then $r\cdot S(p,2)\le p^{2\frac{\log_2r}{4}}$ unless one of the following holds
\begin{enumerate}
\item\label{lemma2:eq4}$p=7$ and $1\le r/p^2\le 6$,
\item\label{lemma2:eq5}$p= 5$ and $1\le r/p^2 \le 16\, 314$,
\item\label{lemma2:eq6}$p\in \{2,3\}$.
\end{enumerate}
\end{lemma}
\begin{proof}
Here $S(p,2)=p+3$. The proof is very similar to the proof of Lemma~\ref{lemma1} and it basically follows from straightforward computations.  When $p\ge 19$, it can be verified that $p^{\log_2(p)/2}>(p+3)p$ and $p^{1/2}>2$. Therefore
\begin{align*}
p^{2\frac{\log_2 r}{4}}&=p^{\frac{\log_{2}p}{2}}\cdot p^{\frac{\log_2(r/p)}{2}}>(p+3)p\cdot p^{\frac{\log_2(r/p)}{2}}\\
&>(p+3)p\cdot 2^{\log_2(r/p)}=(p+3)p\cdot (r/p)=rS(p,2).
\end{align*}

Suppose now $p<19$. If $p\le 3$, then we obtain part~\eqref{lemma2:eq6}. If $p>3$, then $p\in \{5,7,11,13\}$ and part~\eqref{lemma2:eq4} and~\eqref{lemma2:eq5} follow with computer assisted computations by dealing with each case at the time. 
\end{proof}

\begin{lemma}\label{lemma3}
Let $p$ be a prime number and let $r$ be a positive multiple of $p^3$ with $\gcd(p,r/p^3)=1$ and $r/p^3>1$. Then $r\cdot S(p,3)\le p^{3\frac{\log_2r}{4}}$ unless one of the following holds
\begin{enumerate}
\item\label{lemma3:eq5}$p= 5$ and $1\le r/p^3\le 2$,
\item\label{lemma3:eq6}$p\in \{2,3\}$.
\end{enumerate}
\end{lemma}
\begin{proof}
The proof is very similar to the proof of Lemmas~\ref{lemma1} and~\ref{lemma2} and we omit it.
\end{proof}

\begin{lemma}\label{lemma4}
Let $p$ be a prime number and let $r$ be a positive multiple of $p^4$ with $\gcd(p,r/p^4)=1$. Then $r\cdot S(p,4)\le p^{4\frac{\log_2r}{4}}$ unless one of the following holds
\begin{enumerate}
\item\label{lemma4:eq8}$p=3$ and $1\le r/p^4\le 116$,
\item\label{lemma4:eq9}$p=2$.
\end{enumerate}
\end{lemma}
\begin{proof}
The proof is omitted.
\end{proof}

\begin{lemma}\label{lemma5}
Let $p$ be a prime number and let $r$ be a positive multiple of $p^5$ with $\gcd(p,r/p^5)=1$. Then $r\cdot S(p,5)\le p^{5\frac{\log_2r}{4}}$ unless one of the following holds
\begin{enumerate}
\item\label{lemma5:eq8}$p=3$ and $1\le r/p^5\le 11$,
\item\label{lemma5:eq9}$p=2$.
\end{enumerate}
\end{lemma}
\begin{proof}
The proof is omitted.
\end{proof}

\begin{lemma}\label{lemma6}
Let $p$ be a prime number, let $a\ge 6$ be an integer and let $r$ be a positive multiple of $p^a$ with $\gcd(p,r/p^a)=1$. Then $r\cdot S(p,a)\le p^{a\frac{\log_2r}{4}}$ unless one of the following holds
\begin{enumerate}
\item\label{lemma6:eq5}$p= 3$ and $r\in \{729, 1\, 458, 2\, 187, 2\, 916\}$,  
\item\label{lemma6:eq6}$p= 2$.
\end{enumerate}
\end{lemma}
\begin{proof}
The proof is omitted.
\end{proof}

\begin{corollary}\label{corollary1}
Let $R$ be a finite group having order $r$ divisible by $p^a$, where $p$ is a prime number, $a$ is a positive integer and $\gcd(r/p^a,p)=1$. Then either $rS(p,a)\le p^{a\frac{\log_2r}{4}}$ or one of the following holds
\begin{enumerate}
\item\label{corollary1:1}$R$ satisfies Theorem~$\ref{thrm:main}$,
\item\label{corollary1:2}$p\in \{5,7,11,13,17\}$ and $a=1$,
\item\label{corollary1:3}$p=3$ and $a\le 3$,
\item\label{corollary1:4}$p=2$.
\end{enumerate}
\end{corollary}
\begin{proof}
Suppose that $rS(p,a)>p^{a\frac{\log_2r}{4}}$. In particular, $(p,a)$ satisfies the conclusions in Lemmas~\ref{lemma1}--\ref{lemma6}. We consider in turn each of these cases.

When $a=1$, Lemma~\ref{lemma1} holds. In particular, we only need to deal with part~\eqref{lemma1:eq1} and~\eqref{lemma1:eq2} of Lemma~\ref{lemma1}, because when $p\le 17$ we see that parts~\eqref{corollary1:2}--\eqref{corollary1:4} are satisfied. In particular, we only have a finite number of cases to consider. Let $r=p_1^{a_1}\cdots p_\ell^{a_\ell}$ be the factorization of $r$ into distinct prime powers. Let us consider the function
$$f(r):=r^{\ell-1}\prod_{i=1}^\ell S(p_i,a_i).$$
By Section~\ref{group}, if $f(r)\le 7.3722\cdot r^{\log_2r/4+1.5315}$, then Theorem~\ref{thrm:main} holds and hence part~\eqref{corollary1:1} is satisfied. Therefore, we may suppose that $f(r)> 7.3722\cdot r^{\log_2r/4+1.5315}$.
We have implemented this function in a computer and we have checked that no $r$ in the range described in Lemma~\ref{lemma1} parts~\eqref{lemma1:eq1} and~\eqref{lemma1:eq2} satisfies $f(r)>7.3722r^{\log_2r/4+1.5315}$. 

When $a=2$, Lemma~\ref{lemma2} holds. In particular, we only need to deal with part~\eqref{lemma2:eq4} and~\eqref{lemma2:eq5} of Lemma~\ref{lemma2}, because when $p\le 3$ we see that parts~\eqref{corollary1:2}--\eqref{corollary1:4} are satisfied. In particular, we only have a finite number of cases to consider. We have checked that no $r$ in the range described in Lemma~\ref{lemma2} parts~\eqref{lemma2:eq4} and~\eqref{lemma2:eq5} satisfies $f(r)>7.3722\cdot r^{\log_2r/4+1.5315}$. 

When $a=3$, Lemma~\ref{lemma3} holds. In particular, we only need to deal with part~\eqref{lemma3:eq5} of Lemma~\ref{lemma3}, because when $p\le 3$ we see that parts~\eqref{corollary1:3}--\eqref{corollary1:4} are satisfied. In particular, we only have a finite number of cases to consider. We have checked that no $r$ in the range described in Lemma~\ref{lemma3} part~\eqref{lemma3:eq5} satisfies $f(r)>7.3722\cdot r^{\log_2r/4+1.5315}$.

Finally, the cases $a\ge 4$ are analogous.
\end{proof}

\section{Proof of Theorem~\ref{thrm:main}}\label{final}
In this section, we complete the proof of Theorem~\ref{thrm:main} that we have begun in Section~\ref{group}. We argue by induction on $r$. Write $\varepsilon:=1.5315$.

 Let $\mathcal{I}$ be the collection of indices $i\in \{1,\ldots,\ell\}$ with $rS(p_i,a_i)>p_i^{a_i\log_2 r/4}$ and let $\mathcal{P}:=\{p_i\mid i\in\mathcal{I}\}$. From Lemmas~\ref{lemma1}--\ref{lemma6}, we have $\mathcal{P}\subseteq\{2,3,5,7,11,13,17,19,23\}$. Actually, from Corollary~\ref{corollary1},  either Theorem~\ref{thrm:main} holds or $\mathcal{P}\subseteq\{2,3,5,7,11,13,17\}$.
Moreover, if $i\in \mathcal{I}$ with $p_i\in \{5,7,11,13,17\}$, then $a_i=1$.

\smallskip

Assume there exists $i\in \mathcal{I}$ with $p_i\in \{5,7,11,13,17\}$ or with $p_i=3$ and $a_i= 2$ such that ${\bf N}_R(P_i)={\bf C}_R(P_i)$. Let $P_i$ be a Sylow $p_i$-subgroup of $R$ and observe that $P_i$ is cyclic of prime order when $p_i>3$ and $P_i$ is abelian when $p_i=3$. Then, from the Burnside $p$-complement theorem, $R_i$ contains a normal subgroup $N$ with $R=NP_i$ and $N\cap P_i=1$. Thus $R$ is the semidirect product of $N$ with $P_i$. For the moment, let us suppose that $(p_i,a_i)\ne (3,2)$ and we do come back to this case later.
Let $x$ be the number of subgroups of $N$. We claim that $R$ has at most $x(1+r/p_i)$ subgroups. Indeed, the number of subgroups of $R$ contained in $N$ is $x$ and, if $H$ is a subgroup of  $R$ not contained in $N$, then $H=\langle K,P_i^g\rangle$ for some subgroup $K$ of $N$ and for some $g\in R$. Observe that $g$ can be chosen in a transversal of $P_i$ in $R$ and hence we have $r/p_i$ choices for $g$.   Therefore the number of subgroups of $R$ is at most
\begin{align*}
x\left(1+\frac{r}{p_i}\right)&\le 2\cdot 7.3722\cdot \left(\frac{r}{p_i}\right)^{\frac{\log_2(r/p_i)}{4}+\varepsilon}\left(1+\frac{r}{p_i}\right). 
\end{align*}
Moreover,
\begin{align*}
r^{\frac{\log_2r}{4}+\varepsilon}&=\left(\frac{r}{p_i}\right)^{\frac{\log_2r}{4}+\varepsilon}p_i^{\frac{\log_2 r}{4}+\varepsilon}=
\left(\frac{r}{p_i}\right)^{\frac{\log_2(r/p_i)}{4}+\varepsilon}
\left(\frac{r}{p_i}\right)^{\frac{\log_2 p_i}{4}}
p_i^{\frac{\log_2 r}{4}+\varepsilon}\\
&=
\left(\frac{r}{p_i}\right)^{\frac{\log_2(r/p_i)}{4}+\varepsilon}
\left(\frac{r}{p_i}\right)^{\frac{\log_2 p_i}{4}}
r^{\frac{\log_2 p_i}{4}}
p_i^\varepsilon\\
&=
\left(\frac{r}{p_i}\right)^{\frac{\log_2(r/p_i)}{4}+\varepsilon}
r^{\frac{\log_2 p_i}{2}}
p_i^{\varepsilon-\frac{\log_2p_i}{4}}.
\end{align*}
When $p_i\ge 5$, going through the various possibilities for $p_i$, it is not hard to verify that $r^{\log_2p_i/2}>r$ and $p^{\varepsilon-\log_2p_i/4}>2$ and hence
$$r^{\frac{\log_2 p_i}{2}}
p_i^{\varepsilon-\frac{\log_2p_i}{4}}>1+\frac{r}{p_i}$$
and the theorem follows in this case. When $(p_i,a_i)=(3,2)$, $P_i$ is no longer cyclic of prime order. However, $P_i$ has at most $5$ non-identity subgroups. Thus arguing as above, we deduce that the number of subgroups of $R$ is at most $x(1+5\cdot r/9)$.
Now, we may repeat the computations above (with minor modifications) and we obtain that the theorem follows.

\smallskip

For the rest of the argument we may suppose that, for every $i\in\mathcal{I}$ with $p_i\in\{3,5,7,11,13,17\}$, either ${\bf N}_R(P_i)>{\bf C}_R(P_i)$ or $(p_i,a_i)\in \{(3,1),(3,3)\}$. When ${\bf N}_R(P_i)>P_i$, the number of Sylow $p_i$-subgroups of $R$ is at most $|R|/2p_i^{a_i}$. 
Let 
\begin{align*}
\mathcal{J}&:=\{i\in \mathcal{I}\mid p_i\in \{5,7,11,13,17\}\},\\
\mathcal{J}'&:=\{i\in \mathcal{I}\mid p_i\in \{2,3\}\}.
\end{align*}
In the particular case that there exists $i\in\mathcal{I}$ with $p_i=3$ and $a_i=2$, we do include the index $i$ in $\mathcal{J}$ and remove it from $\mathcal{J}'$.

With this slight improvement and with this notation, we may go back to~\eqref{eq:AA} and deduce that $R$ has at most 
\begin{equation}\label{improved}
r^{-1}\prod_{i\in\mathcal{J}'}r\cdot S(p_i,a_i)\prod_{i\in \mathcal{J}}\frac{r}{2}\cdot S(p_i,a_i)\prod_{\substack{i=1\\i\notin \mathcal{I}}}^\ell r\cdot S(p_i,a_i)
\end{equation}
subgroups. Let us call $A$ this product.

Let $i\in\mathcal{J}$ with $a_i=1$. Observe that here we are only excluding the possibility that $p_i=3$ and $a_i=2$. We have
\begin{align*}
\frac{r}{2}\cdot S(p_i,a_i)&=\frac{r}{2}\cdot 2=r=r^{1-\frac{\log_2p_i}{4}}r^{\frac{\log_2 p_i}{4}}=r^{1-\frac{\log_2p_i}{4}}p_i^{a_i\frac{\log_2 r}{4}}.
\end{align*}
In the case that $p_i=3$ and $a_i=2$, we have
\begin{align*}
\frac{r}{2}\cdot S(3,2)&=\frac{r}{2}\cdot 6=3r=3r^{1-2\frac{\log_2p_i}{4}}r^{2\frac{\log_2 p_i}{4}}=3r^{1-2\frac{\log_2p_i}{4}}p_i^{a_i\frac{\log_2 r}{4}}\\
&\le r^{1-\frac{\log_2p_i}{4}}p_i^{a_i\frac{\log_2 r}{4}}.
\end{align*}
The last inequality follows from a  computation and is only valid when $r\ge 16$; however, when $r<16$, the veracity of Theorem~\ref{thrm:main} can be easily checked with a direct inspection.

From the previous paragraph and~\eqref{improved}, we obtain
\begin{align*}
A&\le r^{-1}\prod_{i\in\mathcal{J}'}S(p_i,a_i)\prod_{i\in\mathcal{J}}r^{1-\frac{\log_2p_i}{4}}p_i^{a_i\frac{\log_2r }{4}}\prod_{i\notin\mathcal{I}}p_i^{a_i\frac{\log_2r }{4}}\\
&=r^{-1}\prod_{i\in\mathcal{J}'}rS(p_i,a_i)\cdot r^{\sum_{j\in\mathcal{J}}1-\frac{\log_2p_i}{4}}\cdot\left(\prod_{i\notin \mathcal{J}'}p_i^{a_i}\right)^{\frac{\log_2 r}{4}}.
\end{align*}
The maximum of $\sum_{i\in\mathcal{J}}1-\log_2(p_i)/4$ is $1.5315=\varepsilon$ and is obtained when $\{p_i\mid i\in\mathcal{J}\}=\{ 3,5, 7, 11, 13\}$. 

If $\mathcal{J}'=\emptyset$, then 
$$A\le r^{\frac{\log_2r}{4}-1+\sum_{j\in\mathcal{J}}1-\frac{\log_2p_i}{4}}\le r^{\frac{\log_2r}{4}-1+\varepsilon}.$$

Assume $|\mathcal{J}'|=1$. Let $i\in \mathcal{J}'$. By Lemma~\ref{lemma0}, we have $S(p_i,a_i)\le 7.3722\cdot p_i^{a_i\log_2r/4}$ and hence
$$A\le 7.3722\cdot r^{\frac{\log_2r}{4}+\sum_{j\in\mathcal{J}}1-\frac{\log_2p_i}{4}}\le 7.3722r^{\frac{\log_2r}{4}+\varepsilon}.$$

Finally assume $|\mathcal{J}'|=2$. Thus $\mathcal{J}'=\{1,2\}$, $p_1=2$ and $p_2=3$. Recall that $a_2\in\{1,3\}$. Now, since the index corresponding to the prime $3$ is not in $\mathcal{J}$, the maximum of $\sum_{i\in\mathcal{J}}1-\log_2(p_i)/4$ is $0.9278$ and is obtained when $\{p_i\mid i\in\mathcal{J}\}=\{5, 7, 11, 13\}$. 
When $a_2=3$, it can be verified that
$$rS(3,a_i)\le r^{\varepsilon-0.9278}\cdot 3^{a_i\frac{\log_2r}{4}},$$
for every $r\ge 68$.  Therefore, when $r\ge 68$, using~\eqref{improved}, we get
$$A\le 7.3722\cdot r^{\frac{\log_2r}{4}+\varepsilon}.$$
The veracity of Theorem~\ref{thrm:main} for smaller values can be checked with a computer. 

Finally suppose $a_2=1$. When ${\bf N}_R(P_2)>P_2$, we may refine the factor $rS(3,a_1)=2r$ in~\eqref{improved} with simply $r$. Now
$$r= r^{1-\log_2(3)}\cdot 3^{a_i\frac{\log_2r}{4}}.$$  Therefore, using~\eqref{improved}, we get again
$$A\le 7.3722\cdot r^{\frac{\log_2r}{4}+\varepsilon}.$$
Assume then ${\bf N}_R(P_2)=P_2$. From Burnside $p$-complement theorem, there exists a normal subgroup $N$ of $R$ with $R=NP_i$ and $N\cap P_i=1$. As $N$ has order relatively prime to $N$, by applying the argument above to $N$, we deduce that $N$ has at most
$$7.3722\cdot \left(\frac{r}{3}\right)^{\frac{\log_2(r/3)}{4}+0.9278}$$
subgroups. Now $R$ has at most
$$7.3722\cdot \left(\frac{r}{3}\right)^{\frac{\log_2(r/3)}{4}+0.9278}\left(1+\frac{r}{3}\right)$$
subgroups. It is not hard to verify that this number is at most $7.3722\cdot r^{\log_2r/4+\varepsilon}$.
\end{document}